\documentclass[11pt,a4paper]{article}

\usepackage[round]{natbib}
\usepackage[pdftex]{graphicx}
\usepackage{amssymb,amsfonts,amsmath}
\usepackage{enumerate}
\usepackage{proof}
\usepackage{color}
\usepackage{amsthm}

\usepackage[config]{subfig}

\usepackage{polski}
\usepackage[english]{babel}

\usepackage[pdftex]{graphicx}
\usepackage{amssymb,amsfonts,amsmath}
\usepackage{enumerate}
\usepackage{subfig}
\usepackage{float}
\usepackage{caption}
\usepackage{enumitem}

\usepackage{dsfont}

\usepackage{hyperref}
\usepackage{authblk}

\newtheorem{theorem}{Theorem}

\newtheorem{remark}{Remark}

\newtheorem{lemma}{Lemma}
\newtheorem{corollary}{Corollary}

\newcommand{\field}[1]{\mathbf{#1}}

\newcommand{\bb}{b}

\newcommand{\X}{X}
\newcommand{\x}{x}

\newcommand{\Y}{\field{Y}}

\newcommand{\HH}{H}

\newcommand{\betab}{\beta}

\DeclareMathOperator*{\argmin}{arg \min}

\newcommand{\PP}{P}

\newcommand{\NN}{N}

\newcommand{\Z}{Z}

\newcommand{\Deltab}{\Delta}
\newcommand{\vv}{v}
\newcommand{\RR}{R}
\newcommand{\XX}{\mathbb{X}}
\newcommand{\EE}{E}
\newcommand{\ND}{\mathcal{N}}
\DeclareMathOperator*{\supp}{supp}

\newcommand{\Sigmab}{\Sigma}
\DeclareMathOperator*{\Cor}{Cor}

\title{Selection  consistency of  Lasso-based procedures for misspecified high-dimensional binary  model and  random regressors}
\author[1,2]{\small Mariusz Kubkowski \footnote{Correspondence to: Mariusz Kubkowski, Institute of Computer Science, Polish Academy of Sciences, 5, Jana Kazimierza, 01-248 Warsaw, Poland, e-mail: m.kubkowski@ipipan.waw.pl}}
\author[1,2]{\small Jan Mielniczuk}
\affil[1]{\footnotesize Institute of Computer Science, Polish Academy of Sciences, Poland}
\affil[2]{\footnotesize Warsaw University of Technology, Poland}

\begin{document}
\maketitle

\begin{abstract}
We consider selection of random  predictors for  high-dimensional regression problem
with  binary response for a  general loss function. Important special case is  when the binary model is semiparametric and the response function is misspecified under parametric model fit. Selection for such a scenario aims at recovering the support of the minimizer of the associated risk with large probability. We propose a  two-step selection procedure which consists of screening and ordering predictors by Lasso method  and then selecting a subset of predictors which minimizes  Generalized Information Criterion on the corresponding nested family of models. We prove consistency of the selection method under conditions which allow for much larger number of predictors than number of observations. For the semiparametric case when  distribution of random  predictors satisfies linear regression conditions  the true and the estimated parameters are collinear and their  common support can be  consistently identified.
\end{abstract}



\section{Introduction}
\label{a70sec1}
We consider random variable $(X^{(n)}, Y^{(n)})\in R^{p_n}\times \{0,1\}$ and corresponding response function defined as a posteriori probability $q_n(x)=P(Y^{(n)}=1|X^{(n)}=x)$.
We  adopt triangular scenario and assume that  $n$  copies $X_1^{(n)},\ldots,X_n^{(n)}$ of a random vector   $X^{(n)}$ in $R^{p_n}$ are observed
together with corresponding binary responses $Y_1^{(n)},\ldots,Y_n^{(n)}$. We assume that observations $(X_i^{(n)},Y_i^{(n)}),\,i=1,\ldots,n $	are iid.
Let $X_i^{(n)}=(X_{i1}^{(n)},\ldots, X_{ip_n}^{(n)})'$. Frequently considered scenario is the  sequential  one.  In this case,  when sample size  $n$  increases we    observe  values of  new   
 predictors  additionally to the ones observed earlier. This  is a special case of the above scheme as then $X_i^{(n+1)}=
({X_i^{(n)T}},X_{i,p_n+1},\ldots,X_{i,p_{n+1}})^T$.
In the following we will skip the upper index $n$ if no ambiguity arises. Moreover, we write $q(x)=q_n(x)$.
We assume that coordinates  $X_{ij}$  of $X_i$ are subgaussian $Subg(\sigma_{jn}^2)$ with subgaussianity parameter $\sigma_{jn}^2$ i.e. it holds that 
$$E\exp(tX_{ij})\leq \exp(t^2\sigma_{jn}^2/2)$$ for all $t\in R$. 
 For future reference let
\[ s_n^2=\max_{j=1,\ldots,p_n}\sigma_{jn}^2\]
and assume in the following that 
\begin{equation}
\label{subG1}
 \gamma^2:=\limsup_n s_n^2<\infty. 
\end{equation}
 In the sequential scenario this is equivalent to an assumption that all subgaussianity parameters are bounded from above.
We assume  moreover that $X_{i1},\ldots, X_{i{p_n}}$ are linearly independent in the sense that
their arbitrary linear combination is not constant almost everywhere.
 In the following ${\XX}_n=(X_1,\ldots,X_n)^T$ will denote matrix of experiment of
dimension $n\times p_n$. \\
For the regression  defined above we consider loss function of the form
\begin{equation}
\label{loss}
 l( b, x, y) =\rho({ b}^T{ x},y),  
 \end{equation}
where $\rho: R\times\{0,1\} \to R$ is some  function, $\bb,\x\in\RR^{p_n},$ $y\in\{0,1\}$ and 
$$ R(\bb) = \EE l(\bb,\x,y)$$
is associated risk function for $\bb \in \RR^{p_n}$. Our aim is to determine the support of $\beta^*$, where 
\begin{equation}\label{def_beta_star_without_intercept}
\beta^* ={\rm argmin}_{b\in R^{p_n}}  R(b). 
\end{equation}
Coordinates of $\beta^*$ corresponding to non-zero coefficients will  be called  active predictors and vector  $\beta^*$   a pseudo-true vector.
This terminology stems from the important special case of our general setting: misspecification case of semiparametric model.
Namely, consider a semiparametric model for which response function is given 
in semiparametric form 
\begin{equation}
\label{semiparam}
q({ x})=: q(\beta^T{ x})
\end{equation}
 for some fixed $\beta$ and unknown $q$. When the loss defined in (\ref{loss})  does not coincide with minus  conditional log-likelihood 
 -$E(Y\log q(b^T{ X})+ (1-Y)\log (1-q(b^T{ X}))$
 pertaining to
$q(b^T{ x})$, in particular when fitted  parametric model is given by
a response function $q_0\not\equiv q$ 
then the model is misspecified. Questions of robustness analysis evolve around the interplay between $\beta$ and $\beta^*$, in particular under what conditions  the directions of $\beta$ and $\beta^*$ coincide (cf important contribution in \cite{Brillinger82} and \cite{Ruud83}). \\
In  the paper we consider  properties of $\beta^*$ for a general loss function  and the case of  
misspecified semiparametric model (\ref{semiparam}) as a special case of this setup.
 For $s\subseteq \{1,\ldots,p_n\}$  let $\beta^*(s)$ be defined as in (\ref{def_beta_star_without_intercept}) when minimum is taken over $b$ with support in $s$.
We define
\[ s^*={\rm supp}(\beta^*(\{1,\ldots,p_n\})=\{i\leq p_n:\beta_{i}^*\neq 0\},\]
denote the support of $\beta^*(\{1,\ldots,p_n\})$
with  $\beta^*(\{1,\ldots,p_n\}) =(\beta_{1}^*,\ldots, \beta_{p_n}^*)^T$.

Let $v_{\pi} = (v_{j_1},\ldots,v_{j_k})^T\in \RR^{|\pi|}$ for $v\in \RR^{p_n}$ and $\pi =\{j_1,\ldots,j_k\} \subseteq \{1,\ldots,p_n\}$. Let $\beta^*_{s^*}\in R^{| s^*|}$ be  $\beta^* = \beta^*(\{1,\ldots,p_n\})$ restricted to its support  $s^*$.
Note that if $ s^*\subseteq s$, then  provided projections are unique (see Section 2) we have
\[ \beta^*_{s^*}=\beta^*(s^*)_{s^*}=\beta^*(s)_{s^*}.\]
Moreover, let $$\beta_{min}^*=\min_{i\in s^*} |\beta_{i}^*|.$$
We remark  that   $\beta^*$, $s^*$ and $\beta_{min}^*$  may  depend  on $n$.
Note that when the parametric model is correctly specified i.e. $q(x)=q(\beta^Tx)$ for some $\beta$ with $l$ being an associated loglikelihhood loss and if  $s$ is the support of  $\beta$  then  we have
 $s=s^*$.\\
 For fixed number $p$ of predictors smaller than sample size $n$ statistical consequences of misspecification of a semiparametric regression model have been intensively studied by H. White and his collaborators in the 80s of the last century.  The concept of  projection on the  fitted parametric model  is central to this investigations  which show how  the distribution of maximum likelihood estimator of $\beta^*$ centered by $\beta^*$ changes under misspecification (cf e.g. \cite{White82} and \cite{Vuong89}). However for the case when $n>p$ the maximum likelihood estimator which is  a natural tool for fixed $p<n$ case is ill-defined and a natural question arises what can be estimated and by what means
 in this case.\\
For high-dimensional case one of possible solutions is to consider two-stage methods
in which the first stage results in  screened subset of regressors with cardinality smaller than $n$ and the second stage employs one of known methods for fixed $p$ case. As the set of regressors for the second stage is random the properties of the procedure need to be thoroughly reconsidered.\\
In the paper  first stage of the procedure is  is based on Lasso  estimation
\begin{equation}
\label{Lasso}
\hat\beta_L={\rm argmin}_{\bb\in \RR^{p_n}}\{R_n(\bb) +
\lambda_L\sum_{i=1}^{p_n} |b_i|\}
\end{equation}
where  $b=(b_1,\ldots,b_{p_n})^T$ and the empirical risk  $R_n(\bb)$ is
\begin{equation*}
\label{a70ll}
R_n(\bb) = \sum_{i=1}^n \rho(\bb^T\X_i,Y_i).
\end{equation*}
Parameter $\lambda_L>0$ is Lasso penalty which penalizes  large $l_1$-norms  of
potential candidates for a solution.  Note that  criterion function in (\ref{Lasso}) for $\rho(s,y)=\log(1+\exp(-s(2y-1))$ can be viewed as penalized empirical risk for the logistic loss.
Lasso estimator is thoroughly studied in
the case of the linear model when considered loss is square loss (see e.g. \cite{BuhlmannGeer11} and \cite{HastieEtAl15}
for references and overview of the subject) and some of the papers  treat the case when such model is fitted to $Y$ which is not necessarily
linearly dependent on regressors (cf \cite{BickelEtAl09} ).
In this case regression model is misspecified w.r.t. linear fit.  
However,  similar results are scarse for other scenarios such as logistic fit under misspecification in particular. One of the notable exceptions is \cite{Negahbanetal2012} where behaviour of Lasso estimate is studied for a general loss function including logistic loss for possibly misspecified models.
For a  recent contributions to study of Kullback-Leibler projections on logistic model (which coincide with (\ref{def_beta_star_without_intercept}) for logistic loss)
 and references we
refer to \cite{KubkowskiMielniczuk17a} and \cite{KubkowskiMielniczuk18}. We also refer to \cite{Luetal012}, where asymptotic distribution  of adaptive Lasso is studied under misspecification
in the case of fixed number of deterministic predictors. 
 In the following
we prove approximation result for  Lasso estimator when predictors are random and $\rho$ is a convex Lipschitz function (cf Theorem 1).
An useful  corollary
of it is determination  of sufficient
 conditions under which active predictors can be separated from spurious ones based on the absolute values of corresponding coordinates of Lasso estimator. This makes construction of nested family containing $s^*$ with large probability possible. In the general framework allowing for misspecification we call selection rule $\hat s^*$ consistent if 
$P(\hat s^*= s^*)\to 1$ when $n\to\infty$.\\
In the case of semiparametric model when predictors are eliptically contoured ( e.g. multivariate normal) it is known that $\beta^*=\eta\beta$ i.e. these two vectors are collinear (\cite{LiDuan89}). Thus in  case when $\eta\neq 0$ we have that $s^*$ coincides with  support $s$  of $\beta$ and the selection consistency of two-step procedure  proved in the paper entails  direction and support recovery of $\beta$.\\
The main objective of the paper is to  prove consistency of two-stage selection procedure which  consists of ordering of predictors according to the absolute values of  corresponding Lasso estimators  and then minimization of Generalized Information Criterion GIC) on resulting nested family.This is a variant of SOS (Screening-Ordering-Selection) procedure introduced in \cite{PokarowskiMielniczuk2015} in the case of the linear model, where the ordering of predictors chosen by Lasso was with respect to absolute values of $t$ statistics from the  linear fit based on these predictors. Here we consider a simpler scheme for which both  screening and ordering of regressors  is based on Lasso fit. 
We consider the case when predictors are subgaussian random variables. The stated  results  to the best of our knowledge are not available
for random predictors even when the model is correctly specified. 
For the second stage we assume that the number  of active predictors is bounded by a deterministic sequence $k_n$ tending to infinity and we minimize GIC on family ${\cal M}$ of models with sizes satisfying also this condition. Such  exhaustive search has been proposed in  \cite{ChenChen08} for linear models  and extended to GLMs in 
\cite{ChenChen12}, see also \cite{MielniczukSzymanowski15}. In these papers GIC has been optimised on all possible subsets of regressors with cardinality not exceeding certain constant $k_n$. Such method is feasible for practical purposes only when $p_n$ is small.
 Here we consider a similar setup but with important differences: 
 ${\cal M}$ is a data-dependent small  nested family of models and 
 optimization of GIC is considered in the case when the original model is misspecified.
 The regressors are assumed random and assumptions are  carefully tailored to this case.\\ 
In numerical experiments we study the performance of grid version of  logistic  and linear SOS and compare it to its several lasso-based competitors.\\
The paper is organized as follows. Section 2 contains auxiliaries, including  new useful probability inequalities for empirical risk in the case  of subgaussian random variables (Lemma 2).  In Section 3 we prove a bound on approximation error for Lasso for misspecified logistic model and random regressors (Theorem 1) which yields separation property of Lasso.  In  Theorems 2 and 3 of Section 4 we prove GIC consistency on nested  family, which in particular can be built according to the order of the
Lasso coordinates.  In Corollary 5 we discuss consequences of the proved  semiparametric binary model when distribution of predictors satisfies linear regressions condition.
In Section 5 we  compare the performance of two-stage selection method for two closely related models, one of which is a logistic model and the second one is misspecified.
\section{ Definitions and auxiliary results}

We assume throughout existence and uniqueness of  projection vector $\betab^*$ which has been defined in (\ref{def_beta_star_without_intercept}). 
We consider cones of the form:
\begin{equation}
\mathcal{C}_{\varepsilon} = \{ \Deltab \in \RR^{p_n} \colon ~ ||\Deltab_{s^{*c}}||_1 \le (3+\varepsilon) ||\Deltab_{s^*}||_1 \}, \label{eq_lasso_cone}
\end{equation}
where $\varepsilon>0$, $s^{*c} = \{1,\ldots,p_n\}\setminus s^*$ and $\Deltab_{s^*} = (\Delta_{s^*_1},\ldots,\Delta_{s^*_{|s^*|}})$ for $s^*=\{s^*_1,\ldots,s^*_{|s^*|}\}$. Cones $\mathcal{C}_{\varepsilon}$ are of special importance 
(see Lemma \ref{lemma_lasso_oracle_inclusion}). For  cone $\mathcal{C}_{\varepsilon}$ we define a quantity $\kappa_{\HH}(\varepsilon)$ which can be regarded as generalized minimal eigenvalue of a matrix in high-dimensional setup:
\begin{equation}
\kappa_{\HH}(\varepsilon) = \inf\limits_{\Deltab\in \mathcal{C}_{\varepsilon}\setminus\{0\}} \frac{\Deltab^T \HH \Deltab}{\Deltab^T\Deltab},  \label{eq_lasso_kappa0}
\end{equation}
where $\HH\in\RR^{p_n\times p_n}$ is non-negative definite matrix,
which in the considered context is usually taken as hessian $D^2 R(\betab^*)$ .

Let $R$ and $R_n$ be the risk and the empirical risk defined above. Moreover, we introduce the following notation:
\begin{gather}
W(\bb) = R(\bb) - R(\betab^*),  \\
W_n(\bb) = R_n(\bb) - R_n(\betab^*),  \\
B_p(r) = \{ \Deltab \in \RR^{p_n} \colon ~ ||\Deltab||_p \le r \}, \text{ for } p=1,2,\\
\beta^*_{min} = \min\limits_{i\in s^*} |\beta_i^*|, \\  
S(r) = \sup\limits_{\bb \in \RR^{p_n}:\bb-\betab^* \in B_1(r)} |W(\bb) - W_n(\bb)|. \label{eq_lasso_S}
\end{gather}

We will need the following margin condition in Lemma \ref{lemma_lasso_oracle_inclusion} and Theorem \ref{th_lasso_beta_min_property_general_loss}:
\begin{enumerate}[label = (MC),ref = (MC)]
\item\label{assumpt_lasso_margin_condition} There exist  $\vartheta,\varepsilon,\delta>0$ and non-negative definite matrix $\HH\in\RR^{p_n\times p_n}$ such that for all $\bb$ with $\bb - \betab^* \in \mathcal{C}_{\varepsilon} \cap B_1(\delta)$ we have 
$$R(\bb) - R(\betab^*) \ge \frac{\vartheta}{2} (\bb - \betab^*)^T \HH (\bb - \betab^*).$$
\end{enumerate}
The above condition can be viewed as a weaker version of strong convexity of function $R$ in the restricted neighbourhood of $\betab^*$ (namely in the intersection of ball $B_1(\delta)$ and cone $\mathcal{C}_{\varepsilon}$). We stress the fact that $\HH$ does not need to be positive definite, as in the Section \ref{section_lasso_general} we use \ref{assumpt_lasso_margin_condition} together with stronger conditions than $\kappa_{\HH}(\varepsilon)>0$ which imply that right hand side of inequality in  \ref{assumpt_lasso_margin_condition} is positive. We also do not require here twice differentiability of $R$. We note in particular that condition \ref{assumpt_lasso_margin_condition} is satisfied in the case of logistic loss, $\X$ being bounded random variable and $\HH = D^2 R(\betab^*)$ (see \cite{Fanetal2014supplement}\nocite{Fanetal14} and \cite{Bach2010}). It is also easily seen that that \ref{assumpt_lasso_margin_condition} is satisfied for quadratic loss, $\X$ satisfying $\EE ||\X||_2^2<\infty$ and $\HH = D^2 R(\betab^*)$. Similar condition to \ref{assumpt_lasso_margin_condition} (called Restricted Strict Convexity) was considered in \cite{Negahbanetal2012} for empirical risk $R_n$:
\begin{equation*}
R_n(\betab^*+\Deltab)-R_n(\betab^*) \ge DR_n(\betab^*)^T\Deltab + \kappa_L ||\Deltab||^2 - \tau^2(\betab^*) 
\end{equation*}
for all $\Deltab\in C(3,s^*)$, some $\kappa_L>0$ and tolerance function $\tau$.

Another important assumption, used in the Theorem \ref{th_lasso_beta_min_property_general_loss} and Lemma \ref{lemma_lasso_S_tail} is the Lipschitz property of $\rho:$
\begin{enumerate}[label = (LL),ref = (LL)]
\item\label{assumpt_lasso_rho_lipschitz} $\exists L>0\ \forall b_1,b_2 \in \RR,y\in\{0,1\}\colon\ |\rho(b_1,y) - \rho(b_2,y)| \le L|b_1-b_2|$. 
\end{enumerate}
Let $|w|$ stand for dimension of $w$.
For  the second step of the procedure we consider an arbitrary family $\mathcal{M} \subseteq 2^{\{1,\ldots,p_n\}}$ of models (which are identified with  subsets of $\{1,\ldots,p_n\}$ and may be data-dependent) such that $s^*\in\mathcal{M},\forall w\in\mathcal{M}:\ |w|\le k_n$ a.e. and $k_n\in\NN_{+}$ is some deterministic sequence. We define Generalized Information Criterion (GIC) as:
\begin{equation}\label{eq_GIC_def}
GIC(w) = nR_n(\hat{\betab}(w)) + a_n|w|,
\end{equation}
where 
\begin{equation*}
\hat{\betab}(w)=\argmin\limits_{\bb\in\RR^{p_n}\colon ~ \bb_{w^c} = 0_{|w^c|}} R_n(\bb)
\end{equation*}
is ML estimator for model $w$ and $a_n>0$ is some penalty. Typical examples of $a_n$ include:
\begin{itemize}
\item AIC (Akaike Information Criterion): $a_n = 2$,
\item BIC (Bayesian Information Criterion): $a_n = \log n$,
\item EBIC($d$) (Extended BIC): $a_n = \log n + 2d\log p_n$, where $d>0$. 
\end{itemize} 

We will study properties of $S_k(r)$ for $k=1,2$, where:

\begin{equation}
\label{S_k}
S_k(r) = \sup\limits_{\bb\in D_k: \bb - \betab^* \in B_2(r)} |(W_n(\bb) - W(\bb)|
\end{equation}

and is the maximal absolute value of the centred empirical risk $W_n(\cdot)$ and  
sets $D_k$ for $k=1,2$ are defined as follows:
\begin{gather}
D_1 = \{\bb\in \RR^{p_n}\colon ~ \exists w\in \mathcal{M}\colon ~ |w|\le k_n \wedge s^*\subset w \wedge \supp \bb \subseteq w \}, \\
D_2 = \{\bb\in \RR^{p_n}\colon ~ \supp\bb \subset s^* \}.
\end{gather}

We note that such definitions of $D_i$ for $i=1,2$ guarantee that if $\bb\in D_i$, then $|\supp(\bb-\betab^*)|\le k_n$, what we exploit in Lemma \ref{lemma_lasso_S_tail}. Moreover, in Section \ref{sect_GIC_consistency} we consider the following condition for $\epsilon>0$, $w\subseteq \{1,\ldots,p_n\}$ and some $\theta>0$:
\begin{enumerate}[label = $C_{\epsilon}(w)\colon$,ref = $C_{\epsilon}(w)$]
\item\label{eq_GIC_margin_condition} $R(\bb) - R(\betab^*) \ge \theta ||\bb - \betab^*||_2^2$  for all $\bb\in\RR^{p_n}$ such that $\supp \bb \subseteq w$ and $\bb-\betab^* \in B_2(\epsilon).$
\end{enumerate}

We observe also that the conditions \ref{assumpt_lasso_margin_condition} and \ref{eq_GIC_margin_condition} are not equivalent, as they hold for $\vv = \bb-\betab^*$ belonging to different sets: $B_1(r)\cap \mathcal{C}_{\varepsilon}$ and $B_2(\epsilon)\cap\{\Deltab\in\RR^{p_n}\colon ~ \supp\Deltab\subseteq w\}$, respectively. We note that if the  minimal eigenvalue $\lambda_{min}$ of  matrix $\HH$ in condition \ref{assumpt_lasso_margin_condition} is positive 
and \ref{assumpt_lasso_margin_condition} holds for $\bb-\betab^*\in B_1(r)$ (instead of for $\bb - \betab^* \in \mathcal{C}_{\varepsilon} \cap B_1(r)$) then we have for $\bb-\betab^* \in B_2(r/\sqrt{p_n})\subseteq B_1(r)$:
\begin{equation*}
R(\bb) - R(\betab^*) \ge \frac{\vartheta}{2} (\bb - \betab^*)^T \HH (\bb - \betab^*) \ge \frac{\vartheta \lambda_{min}}{2} ||\bb-\betab^*||_2^2.
\end{equation*}
Furthermore, if  $\lambda_{max}$ is the maximal eigenvalue of $H$
and \ref{eq_GIC_margin_condition} holds for all $\vv=\bb-\betab^*\in B_2(r)$ without restriction on $\supp \bb$, then we have for $\bb -\betab^* \in B_1(r)\subseteq B_2(r)$:
\begin{equation*}
R(\bb) - R(\betab^*) \ge \theta ||\bb - \betab^*||_2^2 \ge \frac{\theta}{\lambda_{max}} (\bb - \betab^*)^T \HH (\bb - \betab^*).
\end{equation*}
Similar condition to \ref{eq_GIC_margin_condition} for empirical risk $R_n$ was considered in \cite[formula (2.1)]{KimJeon2016} in the context of GIC minimization.

It turns out that condition \ref{eq_GIC_margin_condition} together with $\rho(\cdot,y)$ being convex for all $y$ and satisfying Lipschitz condition \ref{assumpt_lasso_rho_lipschitz} are sufficient to establish bounds which ensure GIC consistency for $k_n\ln p_n = o(n)$ and $k_n\ln p_n = o(a_n)$ (see Corollaries \ref{coro_GIC_consistency_supsets} and \ref{coro_GIC_consistency_subsets}).

\begin{lemma}\label{lemma_basic_inequality} (Basic inequality). Let $\rho(\cdot,y)$ be convex function for all $y.$ If for some $r>0$ we have:
\begin{equation*}
u=\frac{r}{r + ||\hat{\betab}_L-\betab||_1}, \quad
\vv = u \hat{\betab}_L + (1-u) \betab^*,
\end{equation*}
then:
$$W(\vv) + \lambda ||\vv - \betab^*||_1  \le S(r) + 2\lambda ||\vv_{s^*} - \betab^*_{s^*}||_1.$$
\end{lemma}

Proof  of  the lemma is moved to Appendix. Quantities $S_k(r)$ are diefined in (\ref{S_k}).
\begin{lemma}\label{lemma_lasso_S_tail}
Let $\rho(\cdot,y)$ be convex function for all $y$ and satisfy Lipschitz condition \ref{assumpt_lasso_rho_lipschitz}.
 Assume that $X_{ij}$ for $j\ge 1$ are subgaussian $Subg(\sigma_{jn}^2)$, where $\sigma_{jn}\le s_n$. Then for $r,t>0$:
\begin{enumerate}
\item\label{lemma_lasso_S_tail_p1} $\PP(S(r)>t) \le \frac{8Lrs_n \sqrt{\log (p_n\vee 2)}}{t \sqrt{n}}$,
\item\label{lemma_lasso_S_tail_p2} $\PP(S_1(r) \ge t) \le \frac{8Lrs_n \sqrt{k_n\ln(p_n \vee 2)}}{t\sqrt{n}} $,
\item\label{lemma_lasso_S_tail_p3} $\PP(S_2(r) \ge t) \le \frac{4Lrs_n\sqrt{|s^*|}}{t\sqrt{n}} $.
\end{enumerate}
\end{lemma}

\begin{proof}
From the Chebyshev inequality (first inequality below), symmetrization inequality (see \cite[Lemma 2.3.1]{VanDerVaartWellner1996}) and Talagrand - Ledoux inequality (\cite[Theorem 4.12]{LedouxTalagrand1991}) we have for $t>0$ and $(\varepsilon_i)_{i=1,\ldots,n}$ being Rademacher variables independent of $(\X_i)_{i=1,\ldots,n}$:
\begin{align}
\PP(S(r)>t) & \le \frac{\EE S(r)}{t}\nonumber\\ 
& \le \frac{2}{tn} \EE \sup\limits_{\bb\in\RR^{p_n}:\bb - \betab^* \in B_1(r)} \left| \sum\limits_{i=1}^{n} \varepsilon_i ( \rho(\X_i^T\bb,Y_i) - \rho(\X_i^T\betab^*,Y_i) ) \right| \nonumber\\
& \le \frac{4L}{tn} \EE \sup\limits_{\bb\in\RR^{p_n}:\bb - \betab^* \in B_1(r)} \left| \sum\limits_{i=1}^{n} \varepsilon_i \X_i^T (\bb - \betab^*) \right|. \label{lemma_lasso_S_tail_eq1}
\end{align}

We observe that $\varepsilon_i X_{ij} \sim Subg(\sigma_{jn}^2).$ Hence, using independence we obtain $\sum\limits_{i=1}^{n} \varepsilon_i X_{ij} \sim Subg(n\sigma_{jn}^2)$ and thus $\sum\limits_{i=1}^{n} \varepsilon_i X_{ij} \sim Subg(ns_n^2).$
Applying H{\"o}lder inequality and  the following inequality (see \cite[Lemma 2.2]{DevroyeLugosi2012}):
\begin{equation}\label{lemma_lasso_S_tail_eq2}
\EE \left|\left| \sum\limits_{i=1}^{n} \varepsilon_i X_{ij} \right|\right|_{\infty} \le \sqrt{n} s_n \sqrt{2\ln(2p_n)} \le 2 s_n \sqrt{n \ln(p_n \vee 2)}
\end{equation}
we have:
\begin{align*}
\frac{4L}{tn} \EE \sup\limits_{\bb\in\RR^{p_n}:\bb - \betab^* \in B_1(r)} \left| \sum\limits_{i=1}^{n} \varepsilon_i \X_i^T (\bb - \betab^*) \right| &\le \frac{4Lr}{t} \EE \max\limits_{j\in\{1,\ldots,p_n\}} \left| \frac{1}{n} \sum\limits_{i=1}^{n} \varepsilon_i X_{ij} \right| \\
&\le \frac{8Lrs_n\sqrt{\log (p_n\vee 2)}}{t\sqrt{n}}.
\end{align*}
From this part \ref{lemma_lasso_S_tail_p1} follows. In the proofs of parts \ref{lemma_lasso_S_tail_p2}-\ref{lemma_lasso_S_tail_p3} first inequalities are identical as in (\ref{lemma_lasso_S_tail_eq1}) with supremums taken on corresponding sets.
Using Cauchy-Schwarz inequality, inequality $||\vv||_2 \le \sqrt{||\vv||_0}||\vv||_{\infty}$, inequality $||\vv_{\pi}||_{\infty} \le ||\vv||_{\infty}$ for $\pi \subseteq \{1,\ldots,p_n\}$ and (\ref{lemma_lasso_S_tail_eq2}) yields:
\begin{align*}
\PP(S_1(r) \ge t) 
& \le \frac{4L}{nt} \EE \sup \limits_{\bb\in D_1\colon  \bb - \betab^* \in B_2(r)} \left| \sum\limits_{i=1}^{n} \varepsilon_i\X_i^T(\bb - \betab^*)\right| \\
&\le \frac{4Lr}{nt} \EE \max\limits_{\pi\subseteq\{1,\ldots,p_n\}, |\pi|\le k_n} \left|\left|\sum\limits_{i=1}^{n} \varepsilon_i \X_{i,\pi} \right|\right|_2 \\
&\le  \frac{4Lr}{nt} \EE \max\limits_{\pi\subseteq\{1,\ldots,p_n\}, |\pi|\le k_n} \sqrt{|\pi|} \left|\left|\sum\limits_{i=1}^{n} \varepsilon_i \X_{i,\pi} \right|\right|_{\infty} \\
&\le \frac{4Lr\sqrt{k_n}}{nt} \EE \left|\left|\sum\limits_{i=1}^{n} \varepsilon_i \X_{i} \right|\right|_{\infty}
 \le \frac{8Lr}{t\sqrt{n}} \sqrt{k_n } s_n \sqrt{\ln(p_n \vee 2)} .
\end{align*}
Similarly for $S_2(r)$, using Cauchy-Schwarz inequality, $||\vv_{\pi}||_{2} \le ||\vv_{s^*}||_{2}$ which is valid for $\pi \subseteq s^*$, definition of $l_2$ norm and inequality $\EE |Z| \le \sqrt{\EE Z^2} \le \sigma$ for $Z \sim Subg(\sigma^2)$, we obtain:   
\begin{align*}
\PP(S_2(r) \ge t)
&\le \frac{4L}{nt} \EE \sup \limits_{\bb\in D_2\colon  \bb - \betab^* \in B_2(r)} \left| \sum\limits_{i=1}^{n} \varepsilon_i\X_i^T(\bb - \betab^*)\right| \\
&\le \frac{4Lr}{nt} \EE \max\limits_{\pi\subseteq s^*} \left|\left|\sum\limits_{i=1}^{n} \varepsilon_i \X_{i,\pi} \right|\right|_2 
 \le \frac{4Lr}{nt} \EE \left|\left|\sum\limits_{i=1}^{n} \varepsilon_i \X_{i,s^*} \right|\right|_2 \\
&\le  \frac{4Lr}{nt} \sqrt{\EE \left|\left|\sum\limits_{i=1}^{n} \varepsilon_i \X_{i,s^*} \right|\right|_2^2} 
 = \frac{4Lr}{nt}\sqrt{ \sum\limits_{j\in s^*}\EE \left(\sum\limits_{i=1}^{n} \varepsilon_i X_{ij} \right)^2 } \\
&\le \frac{4Lr}{\sqrt{n}t} \sqrt{|s^*|} s_n.
\end{align*}
\end{proof}

\section{Properties of Lasso for a general loss function and random predictors}\label{section_lasso_general}
The main Theorem in this section is Theorem \ref{th_lasso_beta_min_property_general_loss}. Idea of the proof is based on fact that if $S(r)$ defined in (\ref{eq_lasso_S}) is sufficiently small, then $\hat{\betab}_L$ lies in a ball $\{\Deltab \in \RR^{p_n} \colon ~ ||\Deltab - \betab^*||_1 \le r\}$ (see Lemma \ref{lemma_lasso_oracle_inclusion}). Using
a tail inequality for $S(r)$ proved in Lemma \ref{lemma_lasso_S_tail} we obtain Theorem \ref{th_lasso_beta_min_property_general_loss}.  Convexity of  $\rho(\cdot,y)$ below is understood as convexity for both $y=0,1$.
\begin{lemma}\label{lemma_lasso_oracle_inclusion}
Let $\rho(\cdot,y)$ be convex function and assume that $\lambda >0.$ Moreover, assume margin condition \ref{assumpt_lasso_margin_condition} with constants $\vartheta,\epsilon,\delta>0$ and some non-negative definite matrix $\HH\in\RR^{p_n\times p_n}$.
If for some $r\in (0,\delta]$ we have $S(r) \le \bar{C} \lambda r$ and $2|s^*|\lambda \le \kappa_{\HH}(\varepsilon)\vartheta\tilde{C} r$, where $\bar{C} = \varepsilon/(8+2\varepsilon) $ and $\tilde{C} = 2/(4+\varepsilon),$ then 
$$||\hat{\betab}_L - \betab^*||_1 \le r.$$
\end{lemma}

\begin{proof}
Let $u$ and $\vv$ be defined as in Lemma \ref{lemma_basic_inequality}.
Observe that $||\vv - \betab^*||_1 \le r/2$ is equivalent to $||\hat{\betab}_L - \betab^*||_1 \le r,$ as the function $f(x) = rx/(x+r)$ is increasing, $f(r) = r/2$ and $f(||\hat{\betab}_L - \betab^*||_1) = ||\vv-\betab^*||_1$. Let $C = 1/(4+\varepsilon).$ We consider two cases:\\
(i) $||\vv_{s^*} - \betab^*_{s^*}||_1 \le Cr$:\\
In this case from the basic inequality (Lemma \ref{lemma_basic_inequality}) we have:
$$||\vv - \betab^*||_1 \le \lambda^{-1} (W(\vv) + \lambda ||\vv - \betab^*||_1) \le \lambda^{-1}S(r) + 2 ||\vv_{s^*} - \betab^*_{s^*}||_1 \le \bar{C}  r + 2Cr = \frac{r}{2}.$$
(ii) $||\vv_{s^*} - \betab^*_{s^*}||_1 > Cr$:\\
Note  that $||\vv_{s^{*c}}||_1 <(1-C)r,$ otherwise we would have $||\vv - \betab^*||_1>r$ which contradicts (\ref{lemma_basic_inequality_intermediate_point_lies_in_ball}) in proof of Lemma 1 (see Appendix).\\
Now we observe that $\vv - \betab^* \in \mathcal{C}_{\varepsilon},$ as we have from definition of $C$ and assumption for this case:
$$||\vv_{s^{*c}}||_1 <(1-C)r = (3+\varepsilon)Cr  < (3+\varepsilon) ||\vv_{s^*} - \betab^*_{s^*}||_1.$$
By inequality between $l_1$ and $l_2$ norms, definition of $\kappa_{\HH}(\varepsilon),$ inequality $ca^2/4+b^2/c\ge ab$ and margin condition \ref{assumpt_lasso_margin_condition} (which holds because $\vv - \betab^*\in B_1(r)\subseteq B_1(\delta)$ in view of (\ref{lemma_basic_inequality_intermediate_point_lies_in_ball})) we conclude that:
\begin{align}
||\vv_{s^*} - \betab^*_{s^*}||_1 & \le   \sqrt{|s^*|} ||\vv_{s^*} - \betab^*_{s^*}||_2 \le \sqrt{|s^*|} ||\vv - \betab^*||_2 \\
&\le   \sqrt{|s^*|} \sqrt{\frac{(\vv - \betab^*)^T \HH (\vv - \betab^*)}{\kappa_{\HH}(\varepsilon)}} \nonumber \\
& \le  \frac{\vartheta (\vv - \betab^*)^T \HH (\vv - \betab^*)}{4\lambda} + \frac{|s^*|\lambda}{\vartheta\kappa_{\HH}(\varepsilon)} \le \frac{W(\vv)}{2\lambda} + \frac{|s^*|\lambda}{\vartheta\kappa_{\HH}(\varepsilon)}. \label{lemma_lasso_oracle_inclusion_ineq1}
\end{align}
Hence from the basic inequality (Lemma \ref{lemma_basic_inequality}) and  inequality  above it follows that:
$$ W(\vv) + \lambda ||\vv - \betab^*||_1  \le S(r) + 2\lambda ||\vv_{s^*} - \betab^*_{s^*}||_1 \le S(r) + W(\vv) + \frac{2|s^*|\lambda^2}{\vartheta\kappa_{\HH}(\varepsilon)}.$$
Subtracting $W(\vv)$ from both sides of above inequality, using assumption on $S,$  the bound on $|s^*|$ and definition of $\tilde{C}$ yields:
$$||\vv - \betab^*||_1 \le \frac{S(r)}{\lambda} + \frac{2|s^*|\lambda}{\vartheta\kappa_{\HH}(\varepsilon)} \le \bar{C} r + \frac{2|s^*|\lambda}{\vartheta\kappa_{\HH}(\varepsilon)} \le (\bar{C} + \tilde{C})r = \frac{r}{2}.$$
\end{proof}

\begin{theorem}\label{th_lasso_beta_min_property_general_loss}
Let $\rho(\cdot,y)$ be convex function for all $y$ and satisfy Lipschitz condition \ref{assumpt_lasso_rho_lipschitz}. Assume that $X_{ij}\sim Subg(\sigma_{jn}^2)$, $\betab^*$  exists and is unique, margin condition \ref{assumpt_lasso_margin_condition} is satisfied for $\varepsilon,\delta,\vartheta>0$, non-negative definite matrix $\HH\in\RR^{p_n\times p_n}$ and let
$$\frac{2|s^*|\lambda}{\vartheta\kappa_{\HH}(\varepsilon)} \le \tilde{C} \min\left\{\frac{\beta^*_{min}}{2}, \delta\right\},$$
where $\tilde{C} = 2/(4+\varepsilon).$
Then:
$$\PP\left(||\hat{\betab}_L - \betab^*||_1 \le \frac{\beta^*_{min}}{2} \right) \ge 1 - \frac{8(8+2\varepsilon) L s_n \sqrt{\log(p_n\vee 2)}}{\varepsilon \lambda \sqrt{n}} .$$
\end{theorem}

\begin{proof}
Let:
$$m = \min\left\{\frac{\beta^*_{min}}{2}, \delta\right\} .$$
Lemmas \ref{lemma_lasso_oracle_inclusion} and \ref{lemma_lasso_S_tail} imply that:
\begin{align*}
\PP\left( ||\hat{\betab}_L - \betab^*||_1 > \frac{\beta_{min}^*}{2}\right) & \le \PP\left( ||\hat{\betab}_L - \betab^*||_1 > m\right) \le  \PP\left(S\left(m\right) > \bar{C} \lambda m\right)  \\
& \le \frac{8(8+2\varepsilon)Ls_n\sqrt{\log (p_n \vee 2)}}{\varepsilon\lambda \sqrt{n}}.
\end{align*}
\end{proof}

\begin{corollary}\label{coro_lasso_separation_property} (Separation property) 
If assumptions of Theorem \ref{th_lasso_beta_min_property_general_loss} are satisfied, $\log p_n = o(\lambda^2 n)$ and $\kappa_{\HH}(\varepsilon) > d$ for some $d,\varepsilon>0$ for large $n$, $|s^*|\lambda=o(\min\{\beta^*_{min},1\}), $ then 
$$\PP\left(||\hat{\betab}_L - \betab^*||_1 \le \frac{\beta_{min}^*}{2}\right)\rightarrow 1.$$ 
Moreover 
$$\PP\left(\max\limits_{i\in s^{*c}} |\hat{\beta}_{L,i}|\, \le \,\min\limits_{i\in s^{*}} |\hat{\beta}_{L,i}|\right) \rightarrow 1.$$
\end{corollary}

\begin{proof}
First part of the corollary follows directly from Theorem \ref{th_lasso_beta_min_property_general_loss}. Now we prove that condition $||\hat{\betab}_L - \betab^*||_1 \le \beta^*_{min}/2$ implies separation property 
$$ \max\limits_{i\in s^{*c}} |\hat{\beta}_{L,i}| \le \min\limits_{i\in s^{*}} |\hat{\beta}_{L,i}|.$$

Observe that for all $j\in\{1,\ldots,p_n\}$ we have:
\begin{equation}\label{coro_lasso_separation_property_ineq1}
\frac{\beta^*_{min}}{2} \ge ||\hat{\betab}_L - \betab^*||_1 \ge |\hat{\beta}_{L,j} - \beta^*_j|.
\end{equation}
If $j\in s^*,$ then using triangle inequality yields:
$$|\hat{\beta}_{L,j}- \beta_j^*| \ge |\beta_j^*| - |\hat{\beta}_{L,j}| \ge \beta^*_{min} - |\hat{\beta}_{L,j}|.$$
Hence from the above inequality and (\ref{coro_lasso_separation_property_ineq1}) we obtain for $j\in s^*$:
$$|\hat{\beta}_{L,j}| \ge \frac{\beta^*_{min}}{2}.$$
If $j\in s^{*c},$ then $\beta^*_j = 0$ and (\ref{coro_lasso_separation_property_ineq1}) takes the form:
$$|\hat{\beta}_{L,j}| \le \frac{\beta^*_{min}}{2} .$$
This ends the proof.
\end{proof}

\section{GIC consistency for a a general loss function and random prdictors}\label{sect_GIC_consistency}
Theorems \ref{th_GIC_consistency_supsets_inequality} and \ref{th_GIC_consistency_subsets_inequality} state probability inequalities related to behaviour of  GIC  on supersets  and on subsets of $s^*$, respectively. Corollaries \ref{coro_GIC_consistency_supsets} and \ref{coro_GIC_consistency_subsets} present asymptotic conditions for GIC consistency in the aforementioned situations. Corollary \ref{coro_SS_consistency_general} gathers conclusions of Theorem \ref{th_lasso_beta_min_property_general_loss} and Corollaries \ref{coro_lasso_separation_property}, \ref{coro_GIC_consistency_supsets} and \ref{coro_GIC_consistency_subsets} to show consistency of SS procedure (see \cite{PokarowskiMielniczuk2015}) in case of subgaussian variables.

\begin{theorem}\label{th_GIC_consistency_supsets_inequality}
Assume that $\rho(\cdot,y)$ is convex, Lipschitz function with constant $L>0$, $X_{ij}\sim Subg(\sigma_{jn}^2),$ condition $C_{\epsilon}(w)$ holds for some $\epsilon,\theta>0$ and for every $w\subseteq\{1,\ldots,p_n\}$ such that $|w|\le k_n$.  
Then for any $r<\epsilon$ we have:
\begin{equation}\label{th_GIC_cosistency_supsets_inequality_ineq}
\PP(\min\limits_{w\in \mathcal{M}: s^*\subset w} GIC(w) \le GIC(s^*)) \le \frac{8L\sqrt{k_n} s_n\sqrt{\ln(p_n \vee 2)}}{\sqrt{n}} \left( \frac{rn}{a_n}  
+ \frac{4}{\theta r} \right).
\end{equation}
\end{theorem}

\begin{proof}
If $s^*\subset w \in \mathcal{M}$ and $\hat{\betab}(w) - \betab^* \in B_2(r)$ then in view of inequalities $R_n(\hat{\betab}(s^*)) \le R_n(\betab^*)$ and $R(\betab^*) \le R(\bb)$  we have:
\begin{align*}
R_n(\hat{\betab}(s^*)) -  R_n(\hat{\betab}(w)) & \le \sup\limits_{\bb\in D_1\colon  \bb - \betab^* \in B_2(r)} (R_n(\betab^*) - R_n(\bb)) \\
& \le  \sup\limits_{\bb\in D_1\colon  \bb - \betab^* \in B_2(r)} ((R_n(\betab^*) - R(\betab^*)) - (R_n(\bb) - R(\bb))) \\
& \le \sup\limits_{\bb\in D_1\colon  \bb - \betab^* \in B_2(r)} |R_n(\bb) - R(\bb) - (R_n(\betab^*) - R(\betab^*))| \\ 
& = S_1(r).
\end{align*}
Note that $a_n(|w|-|s^*|)\ge a_n.$ 
Hence, if we have for some $w\supset s^*$: $GIC(w) \le GIC(s^*)$ then we obtain $nR_n(\hat{\betab}(s^*)) - nR_n(\hat{\betab}(w))) \ge a_n(|w|-|s^*|)$ and from the above inequality we have $S_1(r)\ge {a_n}/{n}$.
Furthermore, if $\hat{\betab}(w) - \betab^* \in  B_2(r)^c$ and $r<\epsilon,$ then consider:
\begin{equation*}
\vv = u \hat{\betab}(w) + (1-u) \betab^*,
\end{equation*}
where $u = r/(r + ||\hat{\betab}(w) - \betab^*||_2)$. Then
\begin{equation*}
||\vv - \betab^*||_2 = u||\hat{\betab}(w) - \betab^*||_2 = r\cdot \frac{ ||\hat{\betab}(w) - \betab^*||_2}{r + ||\hat{\betab}(w) - \betab^*||_2} \ge \frac{r}{2},
\end{equation*}
as function $x/(x+r)$ is increasing with respect to $x$ for $x>0$. Moreover, we have $||\vv - \betab^*||_2 \le r <\epsilon$. Hence, in view of $C_{\epsilon}(w)$ condition we get:
\begin{equation*}
R(\vv) - R(\betab^*) \ge \theta ||\vv - \betab^*||_2^2 \ge \frac{\theta r^2}{4}. 
\end{equation*}
From convexity of $R_n$ we have:
\begin{equation*}
R_n(\vv) \le u(R_n(\hat{\betab}(w)) - R_n(\betab^*)) + R_n(\betab^*) \le R_n(\betab^*).
\end{equation*}
We observe that $\supp \vv \subseteq \supp \hat{\betab}(w) \cup \supp \betab^* \subseteq w$, hence $\vv\in D_1$. Finally, we have:
\begin{align*}
S_1(r) \ge R_n(\betab^*) - R(\betab^*) - (R_n(\vv) - R(\vv)) \ge R(\vv) - R(\betab^*) \ge \frac{\theta r^2}{4}.
\end{align*}
Hence we obtain the following sequence of inequalities:
\begin{multline*}
\PP(\min\limits_{w\in \mathcal{M}: s^*\subset w} GIC(w) \le GIC(s^*)) \\
\le \PP(S_1(r) \ge \frac{a_n}{n}, \forall w\in \mathcal{M}\colon ~ \hat{\betab}(w)-\betab^* \in B_2(r) )\\
+ \PP(\exists w\in \mathcal{M}: s^*\subset w \wedge \hat{\betab}(w)-\betab^* \in B_2(r)^c) 
\le \PP(S_1(r) \ge \frac{a_n}{n} ) + \PP(S_1(r) \ge \frac{\theta r^2}{4}) \\
\le \frac{8Lr\sqrt{n}}{a_n} \sqrt{k_n } s_n \sqrt{\ln(p_n \vee 2)} + \frac{32L}{\theta r\sqrt{n}} \sqrt{k_n } s_n \sqrt{\ln(p_n \vee 2)}.
\end{multline*}
\end{proof}
\begin{corollary}\label{coro_GIC_consistency_supsets}
Assume that $\rho(\cdot,y)$ is convex, Lipschitz function with constant $L>0$, $X_{ij}\sim Subg(\sigma_{jn}^2),$ condition $C_{\epsilon}(w)$ holds for some $\epsilon,\theta>0$ and  for every $w\subseteq\{1,\ldots,p_n\}$ such that $|w|\le k_n$, $k_n\ln (p_n\vee 2) = o(n)$ and $k_n\ln (p_n\vee 2) = o(a_n)$. Then we have
$$\PP(\min\limits_{w\in \mathcal{M}: s^*\subset w} GIC(w) \le GIC(s^*)) \rightarrow 0.$$
\end{corollary}

\begin{proof}
We take: $r_n = C_n \sqrt{\frac{k_n\ln(p_n\vee 2)}{n}}$, where 
$$C_n = \sqrt[4]{\frac{n}{k_n\ln(p_n \vee 2)}} \min \{1 , \sqrt[4]{\frac{a_n}{n}}\}.$$ 
We observe that $C_n \rightarrow +\infty$, $r_n \le \sqrt[4]{\frac{k_n\ln(p_n\vee 2)}{n}} \rightarrow 0$ and 
\begin{equation*}
C_n\frac{k_n\ln(p_n\vee 2)}{a_n} \le \left(\frac{k_n \ln (p_n \vee 2)}{a_n} \right)^{\frac{3}{4}} \rightarrow 0.
\end{equation*} 
In view of Theorem \ref{th_GIC_consistency_supsets_inequality} we have for sufficiently large $n$ such that $r_n<\epsilon$ holds:
\begin{align*}
\PP(\min\limits_{w\in \mathcal{M}: s^*\subset w} GIC(w) \le GIC(s^*)) &\le \frac{8L\sqrt{k_n} s_n\sqrt{\ln(p_n \vee 2)}r_n\sqrt{n}}{a_n}   \\
&\quad + \frac{32L\sqrt{k_n} s_n\sqrt{\ln(p_n \vee 2)}}{\sqrt{n}\theta r_n}  \\
&=  \frac{8LC_nk_n s_n\ln(p_n \vee 2)}{a_n} + \frac{32L s_n}{\theta C_n} \rightarrow 0.
\end{align*}
\end{proof}

The most restrictive condition of Corollary \ref{coro_GIC_consistency_supsets} is $k_n\ln (p_n \vee 2) = o(a_n)$. We note that in the case when $p_n\ge n$ and $k_n = d$, EBIC penalty defined above corresponds to the borderline of this condition.
Theorem \ref{th_GIC_consistency_subsets_inequality} is an analogue of Theorem \ref{th_GIC_consistency_supsets_inequality} for subsets of $s^*$.
\begin{theorem}\label{th_GIC_consistency_subsets_inequality}
Assume that $\rho(\cdot,y)$ is convex, Lipschitz function with constant $L>0$, $X_{ij}\sim Subg(\sigma_{jn}^2),$ condition $C_{\epsilon}(s^*)$ holds for some $\epsilon,\theta>0$ and $8a_n|s^*| \le \theta n\min\{\epsilon^2,\beta^{*2}_{min}\}$. Then we have:
$$
\PP(\min\limits_{w\in \mathcal{M}:w \subset s^*} GIC(w) \le GIC(s^*)) \le \frac{32Ls_n \sqrt{|s^*|}}{\theta \sqrt{n}\min\{\epsilon,\beta^*_{min}\}}.  
$$
\end{theorem}

\begin{proof}
Suppose that for some $w\subset s^*$ we have $GIC(w) \le GIC(s^*)$. This  is equivalent to: 
\begin{equation*}
nR_n(\hat{\betab}(s^*)) - nR_n(\hat{\betab}(w)) \ge a_n(|w| - |s^*|).
\end{equation*}
In view of inequalities $R_n(\hat{\betab}(s^*)) \le R_n(\betab^*)$ and $a_n(|w|-|s^*|)\ge -a_n|s^*|$ we obtain:
\begin{equation*}
nR_n(\betab^*) - nR_n(\hat{\betab}(w)) \ge -a_n|s^*|.
\end{equation*}
Let
$\vv = u\hat{\betab}(w) + (1-u) \betab^*$
for some $u\in [0,1]$ to be specified later. From convexity of $\rho$ we consider:
\begin{equation}\label{th_GIC_consistency_subsets_inequality_1}
nR_n(\betab^*) - nR_n(\vv) \ge nu (R_n(\betab^*) - R_n(\hat{\betab}(w))) \ge - ua_n|s^*| \ge -a_n|s^*|.
\end{equation}

We consider two cases separately:

1) $\beta^*_{min} > \epsilon$.\\
First we observe
\begin{equation}\label{th_GIC_consistency_subsets_inequality_5}
8a_n |s^*| \le \theta \epsilon^2n,
\end{equation}
what follows from our assumption. Let $u = \epsilon/(\epsilon + ||\hat{\betab}(w) - \betab^*||_2)$ and
\begin{equation}\label{th_GIC_consistency_subsets_inequality_6}
\vv = u\hat{\betab}(w) + (1-u) \betab^*.
\end{equation}
Note that $||\hat{\betab}(w) - \betab^*||_2 \ge ||\betab^*_{s^*\setminus w}||_2 \ge \beta_{min}^*$. Then, as function $d(x) = x/(x+c)$ is increasing and bounded from above by $1$ for $x,c>0$, we obtain:
\begin{equation}\label{th_GIC_consistency_subsets_inequality_2}
\epsilon \ge ||\vv - \betab^*||_2 = \frac{\epsilon ||\hat{\betab}(w) - \betab^*||_2}{\epsilon + ||\hat{\betab}(w) - \betab^*||_2} 
\ge \frac{\epsilon\beta^*_{min}}{\epsilon + \beta^*_{min}} > \frac{\epsilon^2}{2\epsilon} =  \frac{\epsilon}{2}.
\end{equation}
Hence, in view of $C_{\epsilon}(s^*)$ condition we have: 
\begin{equation*}
R(\vv) - R(\betab^*) > \theta \frac{\epsilon^2}{4}.
\end{equation*}
Using (\ref{th_GIC_consistency_subsets_inequality_1})-(\ref{th_GIC_consistency_subsets_inequality_6}) and above inequality yields:
\begin{align*}
S_2(\epsilon) \ge R_n(\betab^*) - R(\betab^*) - (R_n(\vv) - R(\vv)) > \theta \frac{\epsilon^2}{4} - \frac{a_n}{n}|s^*| \ge \frac{\theta\epsilon^2}{8}.
\end{align*}
Thus, in view of Lemma \ref{lemma_lasso_S_tail}, we obtain:
\begin{equation}\label{th_GIC_consistency_subsets_inequality_3}
\PP(\min\limits_{w\in \mathcal{M}:w \subset s^*} GIC(w) \le GIC(s^*)) \le \PP\left(S_2(\epsilon) > \frac{\theta\epsilon^2}{8}\right) \le \frac{32L\sqrt{|s^*|} s_n }{\sqrt{n}\theta \epsilon}.
\end{equation} 

2) $\beta^*_{min} \le \epsilon$.\\
In this case we take $u = \beta^*_{min}/(\beta^*_{min} +  ||\hat{\betab}(w) - \betab^*||_2)$ and define $\vv$ as in (\ref{th_GIC_consistency_subsets_inequality_6}). Analogously as in (\ref{th_GIC_consistency_subsets_inequality_2}), we have:
\begin{equation*}
\frac{\beta^*_{min}}{2} \le ||\vv - \betab^*||_2 \le \beta^*_{min}. 
\end{equation*}
Hence, in view of $C_{\epsilon}(s^*)$ condition we have: 
\begin{equation*}
R(\vv) - R(\betab^*) \ge \theta \frac{\beta^{*2}_{min}}{4}.
\end{equation*}
Using (\ref{th_GIC_consistency_subsets_inequality_1}) and above inequality yields:
\begin{align*}
S_2(\beta^*_{min}) \ge R_n(\betab^*) - R(\betab^*) - (R_n(\vv) - R(\vv)) \ge \theta \frac{\beta^{*2}_{min}}{4} - \frac{a_n}{n}|s^*| \ge \frac{\theta}{8} \beta^{*2}_{min}.
\end{align*}
Thus, in view of Lemma \ref{lemma_lasso_S_tail}, we obtain:
\begin{equation}\label{th_GIC_consistency_subsets_inequality_4}
\PP(\min\limits_{w\in \mathcal{M}:w \subset s^*} GIC(w) \le GIC(s^*)) \le \PP\left(S_2(\beta^*_{min}) \ge \frac{\theta}{8} \beta^{*2}_{min}\right) \le \frac{32L\sqrt{|s^*|} s_n }{\sqrt{n}\theta\beta^*_{min}}.
\end{equation} 

By combining (\ref{th_GIC_consistency_subsets_inequality_3}) and (\ref{th_GIC_consistency_subsets_inequality_4}) the theorem follows.
\end{proof}

\begin{corollary}\label{coro_GIC_consistency_subsets}
Assume that loss $\rho(\cdot,y)$ is convex, Lipschitz function with constant $L>0$, $X_{ij}\sim Subg(\sigma_{jn}^2),$ condition $C_{\epsilon}(s^*)$ holds for some $\epsilon,\theta>0$ and $a_n|s^*| = o(n\min\{1,\beta^*_{min}\}^2)$, then
$$\PP(\min\limits_{w\in \mathcal{M}:w \subset s^*} GIC(w) \le GIC(s^*)) \rightarrow 0.$$
\end{corollary}

\begin{proof}
First we observe as $a_n\to\infty$ 
$$a_n|s^*| = o(n\min\{1,\beta^*_{min}\}^2)$$ 
implies 
$$|s^*| = o(n\min\{1,\beta^*_{min}\}^2),$$ 
and thus in view of Theorem \ref{th_GIC_consistency_subsets_inequality} we have
$$\PP(\min\limits_{w\in \mathcal{M}:w \subset s^*} GIC(w) \le GIC(s^*)) \rightarrow 0.$$
\end{proof}

\section{Selection consistency of SS procedure}\label{sect_GIC_SS}
In this section we combine the results of the two previous sections to 
establish consistency of  a two-step SS procedure. It consists in construction of  a nested family of models ${\cal M}$ using magnitude of Lasso coefficients and then finding the minimizer of GIC over this family. As ${\cal M}$ is data-dependent, in order to establish consistency of the procedure we
use Corollaries  \ref{coro_GIC_consistency_supsets} and \ref{coro_GIC_consistency_subsets} in which the minimizer of GIC is considered over {\it all} subsets and supersets of $s^*$.

SS (Screening and Selection) procedure is defined as follows:
\begin{enumerate}
\item Choose some $\lambda>0$.
\item Find $\hat{\betab}_{L} = \argmin\limits_{\bb \in \RR^{p_n}} R_n(\bb) + \lambda ||\bb||_1$.
\item Find $\hat{s}_L = \supp \hat{\betab}_{L} = \{j_1,\ldots,j_k\}$ such that $|\hat{\beta}_{L,j_1}|\ge \ldots \ge |\hat{\beta}_{L,j_k}|>0$ and $j_1,\ldots,j_k\in\{1,\ldots,p_n\}$.
\item Define $\mathcal{M}_{SS} = \{\emptyset, \{j_1\}, \{j_1,j_2\}, \ldots, \{j_1,j_2,\ldots,j_k\}\}$.
\item Find $\hat{s}^* = \argmin\limits_{w\in \mathcal{M}_{SS}} GIC(w)$. 
\end{enumerate}

SS procedure is a modification of SOS procedure in 
\cite{Pokarowskietal} designed for GLMs for which ordering of variables is based on p-values of corresponding significance tests for them. Since additional ordering is omitted in the  proposed modification we compressed the name to SS.

Corollary \ref{coro_SS_consistency_general} and Remark \ref{remark_SS_consistency_general} describe the situations when SS procedure is selection consistent. In  it we use the assumptions imposed in Sections 2 and 3 together with an assumption that support of $s^*$ contains no more than $k_n$ elements, where $k_n$ is some  deterministic sequence of integers.

\begin{corollary}\label{coro_SS_consistency_general}
Assume that $\rho(\cdot,y)$ is convex, Lipschitz function with constant $L>0$, $X_{ij} \sim Subg(\sigma_{jn}^2)$ and $\betab^*$ exists and is unique. If $k_n \in \NN_{+}$ is some sequence, margin condition \ref{assumpt_lasso_margin_condition} is satisfied for some $\vartheta,\delta,\varepsilon>0$, condition $C_{\epsilon}(w)$ holds for some $\epsilon,\theta>0$ and for every $w\subseteq \{1,\ldots,p_n\}$ such that $|w|\le k_n$, $\mathcal{M}_{SS}$ is nested family constructed in the step 4 of SS procedure and the following conditions are fulfilled:
\begin{itemize}
\item $|s^*|\le k_n$,
\item $\PP(\forall w \in \mathcal{M}_{SS}: |w| \le k_n) \rightarrow 1$, 
\item $\liminf\limits_{n} \kappa_{\HH}(\varepsilon)>0$ for some $\varepsilon>0$, where $\HH$ is non-negative definite matrix and $\kappa_{\HH}(\varepsilon)$ is defined in (\ref{eq_lasso_kappa0}),
\item $\log(p_n) = o(n\lambda^2)$,
\item $k_n\lambda =o(\min\{\beta^*_{min},1\})$,
\item $k_n\log p_n = o(n)$,
\item $k_n\log p_n = o(a_n)$,
\item $a_nk_n = o(n \min\{\beta^*_{min},1\}^2)$,
\end{itemize}
then for SS procedure we have
\begin{equation*}
\PP(\hat{s}^* = s^*) \rightarrow 1.
\end{equation*}
\end{corollary}

\begin{proof}
In view of Corollary \ref{coro_lasso_separation_property}  it follows from separation property (\ref{coro_lasso_separation_property_ineq1}) we obtain $\PP(s^*\in \mathcal{M}_{SS}) \rightarrow 1$. Let:
\begin{gather*}
A_1 = \{\min_{w\in\mathcal{M}_{SS}: w\supset s^*, |w|\le k_n} GIC(w) \le GIC(s^*)\}, \\
A_2 = \{\min_{w\in\mathcal{M}_{SS}: w\supset s^*, |w|> k_n} GIC(w) \le GIC(s^*)\}, \\
B = \{\forall w \in \mathcal{M}_{SS}: |w| \le k_n\}.
\end{gather*} 

Then we have again from the fact that $A_2\cap B = \emptyset$, union inequality and Corollary \ref{coro_GIC_consistency_supsets}:
\begin{align}
\PP(\min_{w\in\mathcal{M}_{SS}: w\supset s^*} GIC(w) \le GIC(s^*)) &= \PP(A_1\cup A_2)= \PP(A_1 \cup (A_2\cap B^c))\nonumber\\ 
&\le  \PP(A_1) + \PP(B^c) \rightarrow 0.\label{coro_SS_consistency_general_eq_supsets}
\end{align}

In the analogous way, using $|s^*|\le k_n$ and Corollary \ref{coro_GIC_consistency_subsets} yields:
\begin{equation}\label{coro_SS_consistency_general_eq_subsets}
\PP(\min_{w\in\mathcal{M}_{SS}: w\subset s^*} GIC(w) \le GIC(s^*)) \rightarrow 0.
\end{equation}

Now, observe that in view of definition of $\hat{s}^*$ and union inequality:
\begin{multline*}
\PP(\hat{s}^*=s^*) = \PP(\min_{w\in\mathcal{M}_{SS}: w\neq s^*} GIC(w) > GIC(s^*))\\
\ge 1 - \PP(\min_{w\in\mathcal{M}_{SS}: w\subset s^*} GIC(w) \le GIC(s^*)) \\ 
- \PP(\min_{w\in\mathcal{M}_{SS}: w\supset s^*} GIC(w) \le GIC(s^*)).
\end{multline*}

Thus $\PP(\hat{s}^*=s^*)\rightarrow 1$ in view of above inequality, (\ref{coro_SS_consistency_general_eq_supsets}) and (\ref{coro_SS_consistency_general_eq_subsets}).
\end{proof}

Consider now  the case  of semiparametric model defined in (\ref{semiparam}). Then it is known (cf \cite{Brillinger82}, \cite{Ruud83} and \cite{LiDuan89})  that provided $X$ has  regressions satisfying
\begin{equation}
\label{lin_reg}
E(X|\beta^TX)= u_0 + u\beta^TX,
\end{equation}
where $\beta$ is the true parameter, then $\beta^*=\eta\beta$ and $\eta\neq 0$ if ${\rm Cov}(Y,X)\neq 0$. The linear regressions condition (\ref{lin_reg}) is satisfied e.g. by eliptically contoured distribution, in particular by multivariate normal. We refer  also to \cite{KubkowskiMielniczuk17a})
for discussion and up-to date references to this problem.
In the discussed case  $s^*=s$ and we can state the following result.
\begin{corollary}
\label{Ruud2}
Assume that (\ref{lin_reg}) and assumption of Corollary \ref{coro_SS_consistency_general} are satisfied.
Then $P(\hat s^*=s)\to 1$.
\end{corollary}

\begin{remark}\label{remark_SS_consistency_general}
If $p_n = O(e^{cn^{\gamma}})$ for some $c>0$, $\gamma \in (0,1/2)$, $\xi\in(0,0.5-\gamma)$, $u\in(0,0.5-\gamma-\xi)$, $k_n = O(n^{\xi})$, $\lambda = C_n \sqrt{\log (p_n)/n}$, $C_n = O(n^{u})$, $C_n \rightarrow +\infty$, $n^{-\frac{\gamma}{2}}=O(\beta^*_{min})$, $a_n=dn^{\frac{1}{2}-u}$, then assumptions imposed on asymptotic behaviour of parameters in Corollary \ref{coro_SS_consistency_general} are satisfied.
\end{remark}

\begin{remark}
We note that in order to apply Corollary \ref{coro_SS_consistency_general} to two-step procedure based on Lasso it is required that $|s^*|\le k_n$ and that the support of Lasso estimator with probability tending to $1$ contains no more than $k_n$ elements. Some results bounding $|\supp \hat{\betab}_L|$ are available for deterministic $\X$ (see  \cite{Huangetal2008}) and for random $\X$ (see \cite{Tibshirani2013}), but they are too weak to be useful for EBIC penalties. The other possibility to prove consistency of two-step procedure is to modify it in the first step by using thresholded Lasso (see \cite{Zhou2010}) corresponding to $k_n'$ largest Lasso coefficients where $k_n'\in \NN$ is such that $k_n =o(k_n')$.
\end{remark}

\section{Numerical experiments}
\subsection{Selection procedures}
In performed simulations we have implemented modifications of SS procedure introduced in Section \ref{sect_GIC_SS}, as the original procedure is defined for a single $\lambda$ only. In practice it is generally more convenient to consider some sequence $\lambda_1>\ldots>\lambda_m>0$ instead of only one $\lambda$ in the first step in order to avoid choosing 'the best' $\lambda$.  For  the chosen sequence $\lambda_1,\ldots,\lambda_m$, we construct  corresponding families $\mathcal{M}_1,\ldots,\mathcal{M}_m$ analogously to  $\mathcal{M}$ in the step 4 of SS procedure. Thus we arrive here at the following SSnet procedure, which is the modification of SOSnet procedure in \cite{Pokarowskietal}. Below $\tilde b$ is a vector $b$ with first coordinate corresponding to intercept omitted, $b=(b_0,{\tilde \bb}^T)^T)$:
\begin{enumerate}
\item Choose some $\lambda_1>\ldots>\lambda_m>0$.
\item Find $\hat{\betab}_{L}^{(i)} = \argmin\limits_{\bb \in \RR^{p_n+1}} R_n(\bb) + \lambda_i ||\tilde\bb||_1$ for $i=1,\ldots,m$.
\item Find $\hat{s}_L^{(i)} = \supp \hat{\tilde\betab}_L^{(i)}  = \{j_1^{(i)},\ldots,j_{k_i}^{(i)}\}$ where $j_1^{(i)},\ldots,j_{k_i}^{(i)}$ are such that $|\hat{\beta}_{L,j_1^{(i)}}^{(i)}|\ge \ldots \ge |\hat{\beta}_{L,j_k^{(i)}}^{(i)}|>0$ for $i=1,\ldots,m$.
\item Define $\mathcal{M}_i = \{ \{j_1^{(i)}\}, \{j_1^{(i)},j_2^{(i)}\}, \ldots, \{j_1^{(i)},j_2^{(i)},\ldots,j_{k_i}^{(i)}\}\}$ for $i=1,\ldots,m$.
\item Define $\mathcal{M} = \{\emptyset\} \cup\bigcup\limits_{i=1}^{m} \mathcal{M}_i$.
\item Find $\hat{s}^* = \argmin\limits_{w\in \mathcal{M}} GIC(w)$, where 
$$GIC(w) = \min\limits_{\bb\in\RR^{p_n+1}: \supp \tilde\bb \subseteq w} nR_n(\bb) + a_n (|w|+1).$$
\end{enumerate} 

Instead of constructing families $\mathcal{M}_i$ for each $\lambda_i$ in SSnet procedure, $\lambda$ can be chosen by cross-validation using 1SE rule (see \cite{Friedmanetal2010}) and then proceed as in SS procedure. We  call this procedure SSCV.

The last procedure considered has been introduced by Fan and Tang in \cite{FanTang013} and is Lasso procedure with  penalty parameter $\hat\lambda$  chosen in a data-dependent way  as  for SSCV. Namely,  it is  the minimizer of GIC criterion with $a_n = \log(\log n) \cdot \log p_n$  for which ML estimator has been replaced by Lasso estimator with penalty $\lambda$. Once $\hat\beta_L(\hat\lambda_L)$ is calculated  then $\hat s^*$ is defined as its support.
The  procedure is called LFT in the sequel.
%

We list below versions of the above procedures along with R packages, which were used to choose sequence $\lambda_1,\ldots,\lambda_m$ and computation of Lasso estimator. The following packages were chosen based on selection performance after initial tests for each loss and procedure:
\begin{itemize}
\item SSnet with logistic or quadratic loss: \texttt{ncvreg},
\item SSCV or LFT with logistic or quadratic loss: \texttt{glmnet},
\item SSnet, SSCV or LFT with Huber loss (cf \cite{YiHuang2017}): \texttt{hqreg}.
\end{itemize}

The following functions which were used to optimize $R_n$ in GIC minimization step for each loss:
\begin{itemize}
\item logistic loss: \texttt{glm.fit} (package \texttt{stats}),
\item quadratic loss: \texttt{.lm.fit} (package \texttt{stats}),
\item Huber loss: \texttt{rlm} (package \texttt{rlm}).
\end{itemize}

%
Before applying investigated procedures, each column of matrix 
$$\XX = (X_1,\ldots,X_n)^T$$ 
was standardized as $\hat{\betab}_L$ depends on scaling of predictors. We set length of $\lambda_i$ sequence to $m=20$. Moreover, in all  procedures we considered only $\lambda_i$ for which $|\hat{s}_L^{(i)}|\le n$. It is due to the fact that when $|\hat{s}_L^{(i)}|> n$ Lasso  and ML solutions are not unique (see \cite{RossetZhuHastie2004}, \cite{Tibshirani2013}). For Huber loss we set parameter $\delta = 1/10$ (see \cite{YiHuang2017}). Number of folds in SSCV was set to $K=10$.

Each simulation run  consisted of $L$ repetitions, during which samples $\XX_k=(\X_{1}^{(k)},\ldots,\X_{n}^{(k)})^T$ and $\Y_k = (Y_1^{(k)},\ldots,Y_n^{(k)})^T$ were generated for $k=1,\ldots,L$. For $k$-th sample $(\XX_k,\Y_k)$  estimator $\hat{s}_k^*$ of set of active predictors is obtained by a given procedure as the support of $\hat{\tilde\betab}(\hat{s}_k^*)$, where
$$\hat{\betab}(\hat{s}_k^*)  = (\hat{\betab}_0(\hat{s}_k^*), \hat{\tilde\betab}(\hat{s}_k^*)^T)^T = \argmin_{\bb\in \RR^{p_n+1}} \frac{1}{n} \sum\limits_{i=1}^{n} \rho(\bb^T\X_i^{(k)},Y_i^{(k)})$$ 
is ML estimator for $k$-th sample. 
We denote by
$\mathcal{M}^{(k)}$ is the family $\mathcal{M}$ obtained by a given procedure for $k$-th sample.

In our numerical experiments we have computed the following measures of selection performance which gauge co-direction of  true parameter $\beta$ and $\hat\beta$ and the interplay between $s^*$ and $\hat s^*$:

\begin{itemize}
\item $ANGLE = \frac{1}{L} \sum\limits_{k=1}^{L} \arccos |\cos \angle (\tilde\betab_0,\hat{\tilde\betab}(\hat{s}_k^*))|$, where 
$$\cos \angle (\tilde\betab,\hat{\tilde\betab}(\hat{s}_k^*)) = \frac{\sum\limits_{j=1}^{p_n} \beta_j \hat{\beta}_j(\hat{s}_k^*)}{||\tilde\betab||_2 ||\hat{\tilde\betab}(\hat{s}_k^*)||_2 } $$ 
and we let $\cos \angle (\tilde\betab,\hat{\tilde\betab}(\hat{s}_k^*)) = 0$, if 
$||\tilde\betab||_2 ||\hat{\tilde\betab}(\hat{s}_k^*)||_2  = 0$,
\item $P_{inc} = \frac{1}{L} \sum\limits_{k=1}^{L} I(s^*\in\mathcal{M}^{(k)})$,
\item $P_{equal} = \frac{1}{L}\sum\limits_{k=1}^{L}  I(\hat{s}_k^* = s^*)$.
\item $P_{supset} = \frac{1}{L}\sum\limits_{k=1}^{L} I(\hat{s}_k^* \supseteq s^*)$.
\end{itemize} 

%
%
\subsection{Regression models considered}
\subsubsection{Model M1}\label{sect_numerical_experiments_M1_setup}
In order to investigate behaviour of two-step procedure under misspecification we considered two similar models with different sets of predictors. As sets of predictors differ, this results in correct specification of the first model (model M1) and misspecification of the second (Model M2).\\
Namely, we generated $n$ observations $(\X_i,Y_i)\in\RR^{p+1}\times\{0,1\}$ for $i=1,\ldots,n$ such that:
\begin{gather*}
X_{i0} = 1,X_{i1} = Z_{i1}, X_{i2} = Z_{i2}, X_{ij} = Z_{i,j-7} \text{ for } j=10,\ldots,p,\\
X_{i3} = X_{i1}^2, X_{i4} = X_{i2}^2, X_{i5} = X_{i1}X_{i2},\\
X_{i6} = X_{i1}^2X_{i2}, X_{i7} = X_{i1}X_{i2}^2, X_{i8}=X_{i1}^3,X_{i9}=X_{i2}^3,
\end{gather*}
where $\Z_i = (Z_{i1},\ldots,Z_{ip})^T \sim \ND_p(0_p,\Sigmab)$, $\Sigmab = [\rho^{|i-j|}]_{i,j=1,\ldots,p}$ and $\rho\in (-1,1)$.
We consider response function $q(x) = q_L(x^3)$ for $x\in\RR$, $s=\{1,2\}$ and $\betab_s = (1,1)^T$. This means that:
\begin{align*}
\PP(Y_i = 1 | \X_i = \x_i) &= q(\betab_s^T\x_{i,s}) = q(x_{i1}+x_{i2}) = q_L((x_{i1}+x_{i2})^3) \\
&= q_L(x_{i1}^3+x_{i2}^3+3x_{i1}^2x_{i2}+3x_{i1}x_{i2}^2) \\
&= q_L(3x_{i6}+3x_{i7}+x_{i8}+x_{i9}).
\end{align*}   
We observe that the above binary model is well specified with respect to family of  fitted logistic models. Hence $s^* = \{6,7,8,9\}$ and $\betab^*_{s^*} = (3,3,1,1)^T$ are respectively set of active predictors and non-zero coefficients of projection onto family of logistic models. 

We considered the following parameters in numerical experiments: $n=500, p =150,\rho \in\{-0.9+0.15\cdot k\colon ~ k = 0,1,\ldots,12\}$ and $L=500$ - number of generated data sets for each combination of parameters. We investigated procedures SSnet, SSCV and LFT using logistic, quadratic and Huber (cf \cite{YiHuang2017}) loss functions. For procedures SSnet and SSCV we used GIC penalties with:
\begin{itemize}
\item $a_n = \log n$ (BIC),
\item $a_n = \log n + 2\log p_n$ (EBIC1).
\end{itemize}

\subsubsection{Model M2}\label{sect_numerical_experiments_M2_setup}
We generated $n$ observations $(\X_i,Y_i)\in\RR^{p+1}\times\{0,1\}$ for $i=1,\ldots,n$ such that $\X_i = (X_{i0},X_{i1},\ldots,X_{ip})^T$ and $(X_{i1},\ldots,X_{ip})^T \sim \ND_p(0_p,\Sigmab)$, $\Sigmab = [\rho^{|i-j|}]_{i,j=1,\ldots,p}$ and $\rho\in (-1,1)$.
Response function is  $q(x) = q_L(x^3)$ for $x\in\RR$, $s=\{1,2\}$ and $\betab_s = (1,1)^T$. This means that:
\begin{equation*}
\PP(Y_i = 1 | \X_i = \x_i) = q(\betab_s^T\x_{i,s}) = q(x_{i1}+x_{i2}) = q_L((x_{i1}+x_{i2})^3) 
\end{equation*}   
This model in comparison to the one presented in Section \ref{sect_numerical_experiments_M1_setup} does not contain monomials of $X_{i1}$ and $X_{i2}$ of degree higher than $1$ in its set of predictors. We observe that this binary model is missspecified with respect to fitted family of logistic models, because $q(x_{i1}+x_{i2})\not\equiv q_L(\betab^T\x_i)$ for any $\betab\in\RR^{p+1}$. However, in this case linear regressions condition  (\ref{lin_reg}) is satisfied for $ \X$, as it follows normal distribution (see \cite{KubkowskiMielniczuk18},\cite{LiDuan89}) . Hence in view of Proposition 3.8 in \cite{KubkowskiMielniczuk17a} we have $s^*_{log} = \{1,2\}$ and $\betab^*_{log,s^*_{log}} = \eta(1,1)^T$ for some $\eta>0$.
Parameters $n,p,\rho$ as well as $L$ were chosen as for model M1.


\subsubsection{Results for models M1 and M2}
We look first at behaviour of $P_{inc}$, $P_{equal}$ and $P_{supset}$ for the considered procedures.
We observe that values of $P_{inc}$ for SSCV and SSnet are close to $1$ for low correlations in  model M2  for every tested loss (see Figure \ref{fig_num_exp_M1_M2_P_inc}). In model M1  $P_{inc}$ attains the largest values for SSnet procedure and logistic loss for low correlations - this is due to the fact that in most cases corresponding family $\mathcal{M}$ is the largest among the families created by considered procedures. $P_{inc}$ is close to $0$ in model M1 for quadratic and Huber loss, what results in low values of the remaining indices. This may be due to strong dependences between predictors in model M1, note e.g. that we  have  $\Cor (X_{i1},X_{i8}) = 3/\sqrt{15} \approx 0.77$. It is seen that in model M1 inclusion probability $P_{inc}$ is much lower than in model M2 (except for negative correlations). It it also seen that $P_{inc}$ for SSCV is larger than for LFT and LFT fails with respect to $P_{inc}$ in M1. 

In  model M1   the largest values $P_{equal}$ are attained  for SSnet with BIC penalty, then for SSCV with EBIC1 penalty (see Figure \ref{fig_num_exp_M1_M2_P_equal}). In the model M2 $P_{equal}$ is close to $1$ for SSnet and SSCV with EBIC1 penalty and was much larger than $P_{equal}$ for the corresponding versions using BIC penalty. We also note  that choice of loss was relevant only for larger correlations. These results confirm theoretical result of Theorem 2.1 in \cite{LiDuan89} which show that collinearity holds for broad class of loss function. We observe also that although in the model M2 remaining procedures do not select $s^*$ with high probability, they select  its superset, what is indicated by values of $P_{supset}$ (see Figure \ref{fig_num_exp_M1_M2_P_supset}). This analysis is confirmed by an analysis of $ANGLE$ measure (see Figure \ref{fig_num_exp_M1_M2_angle_sparse}), which attains values close to $0$, when $P_{supset}$ is close to $1$. Low values of $ANGLE$ measure mean that estimated vector $\hat{\tilde\betab}(\hat{s}_k^*)$ is approximately proportional to $\tilde\betab$, what was the case for M2 model, where  normal predictors satisfy linear regressions condition. Note that  the angles of $\hat{\tilde\betab}(\hat{s}_k^*)$ and $\tilde\betab^*$ in M1 significantly  differ  despite the fact that M1 is well specified. Also, for the best performing procedures in both models and {\it any} loss considered, $P_{equal}$ was much larger in M2 than in M1, despite the fact that the latter is correctly specified.  This shows that choosing a simple misspecified model which retains crucial characteristics of
the well specified large model instead of the latter might be beneficial.


In model M1 procedures with BIC penalty performed better than those with EBIC1 penalty, however the gain for $P_{equal}$ was much smaller than the gain when using EBIC1 in M2. LFT procedure performed poorly in model M1 and reasonably well in model M2. The overall winner in both models is SSnet. SSCV performs only slightly worse than SSnet in M2 but performs significantly worse in M1.

Analysis of computing times of 1st and 2nd stage of each procedure shows that SSnet procedure creates large families $\mathcal{M}$ and  GIC minimization becomes computationally intensive. We also observe that the first stage for SSCV is more time consuming  than for SSnet, what is caused by multiple fitting of Lasso in cross-validation. However, SSCV is much faster than SSnet in the second  stage.\\
We conclude that in the considered experiments SSnet with EBIC1 penalty works the best in most cases, however even for the winning procedure strong dependence of predictors results in deterioration of its performance. It is  also clear from our experiments that a choice of GIC penalty is crucial for its performance. Modification of SS procedure which would perform satisfactorily for large correlations is still an open problem.
\begin{figure}[H]
\centering
\includegraphics[width = \textwidth]{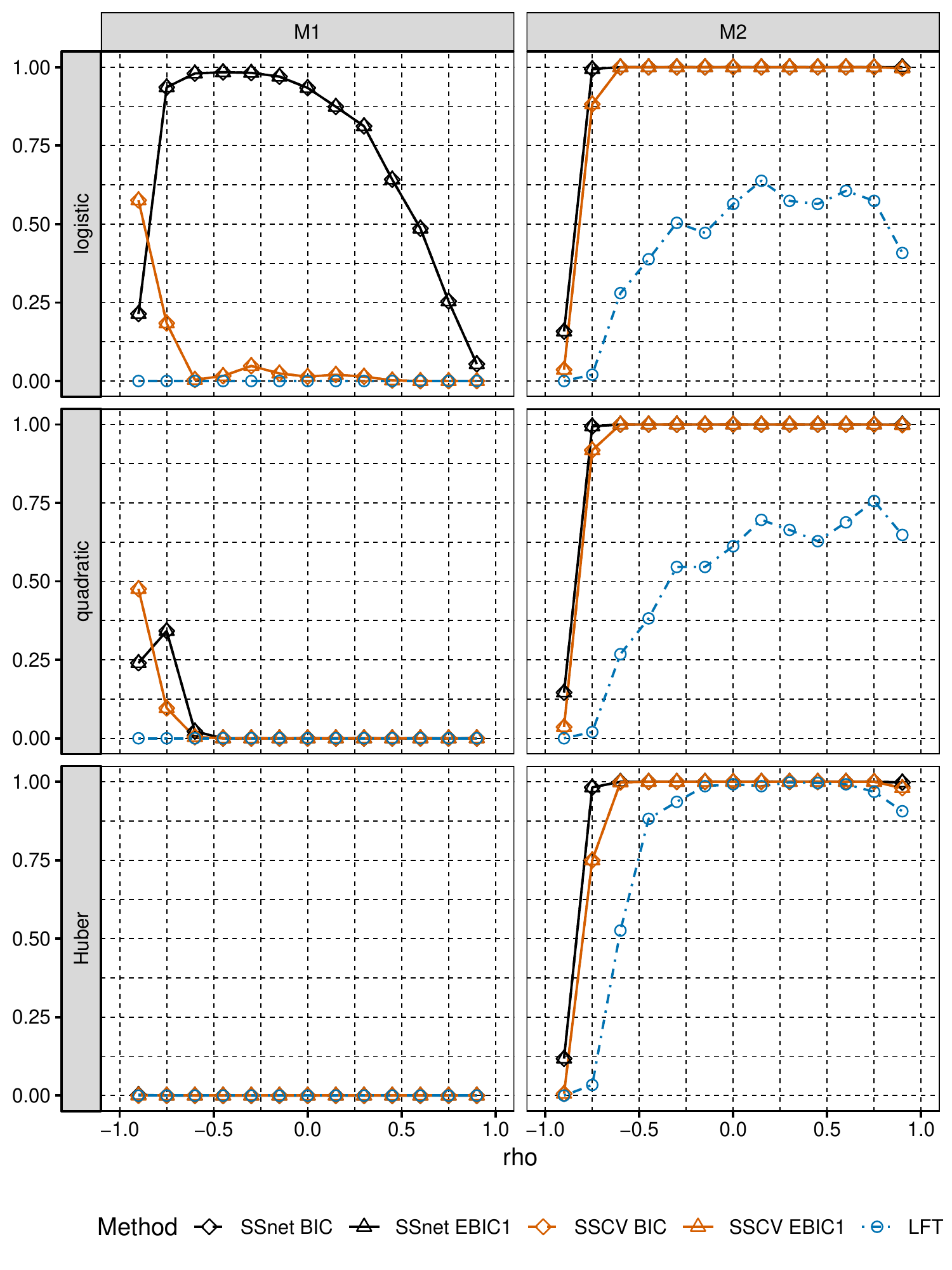}
\caption{$P_{inc}$ for models M1 and M2}
\label{fig_num_exp_M1_M2_P_inc} 
\end{figure}

\begin{figure}[H]
\centering
\includegraphics[width = \textwidth]{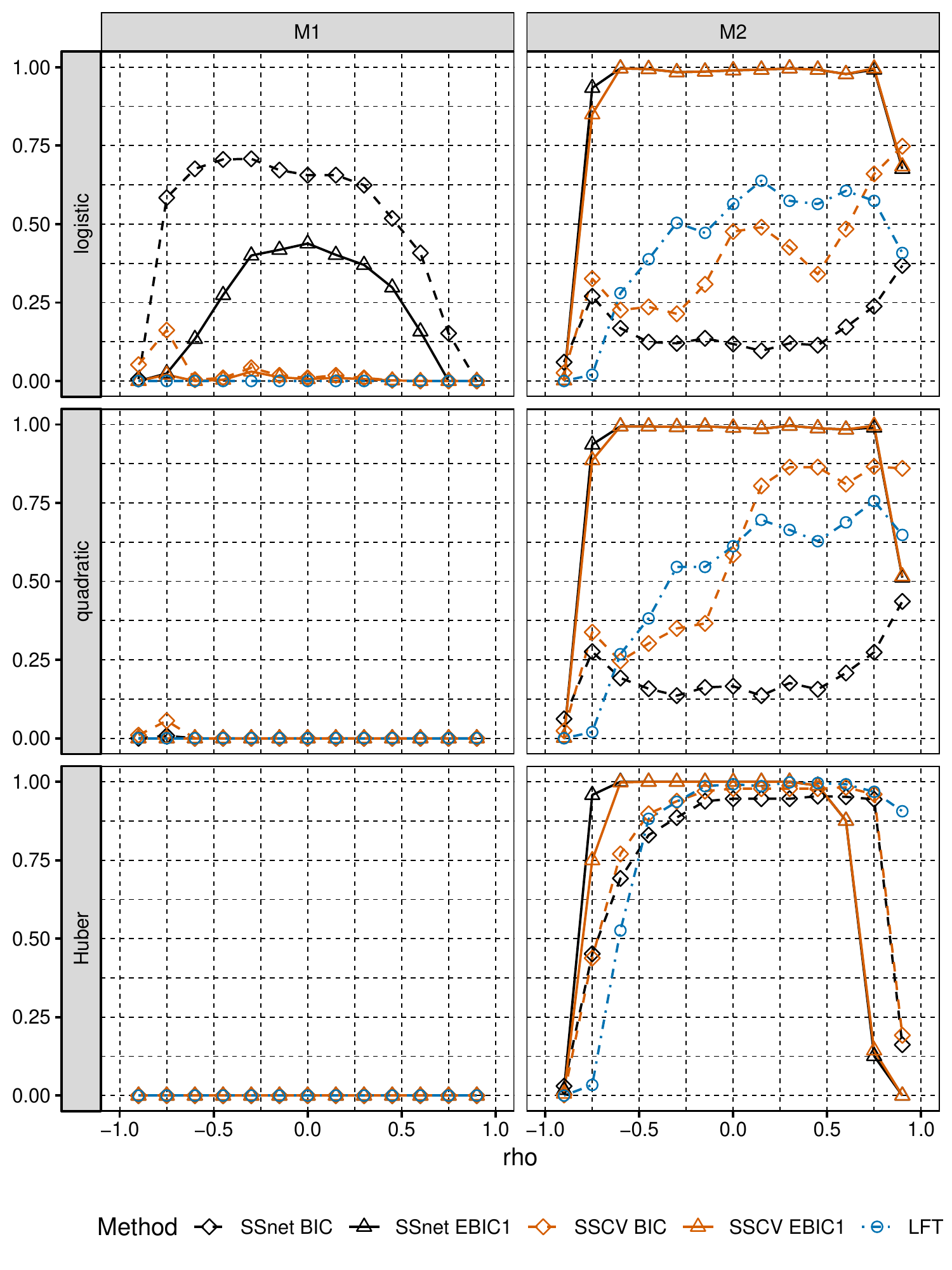}
\caption{$P_{equal}$ for models M1 and M2}
\label{fig_num_exp_M1_M2_P_equal} 
\end{figure}

\begin{figure}[H]
\centering
\includegraphics[width = \textwidth]{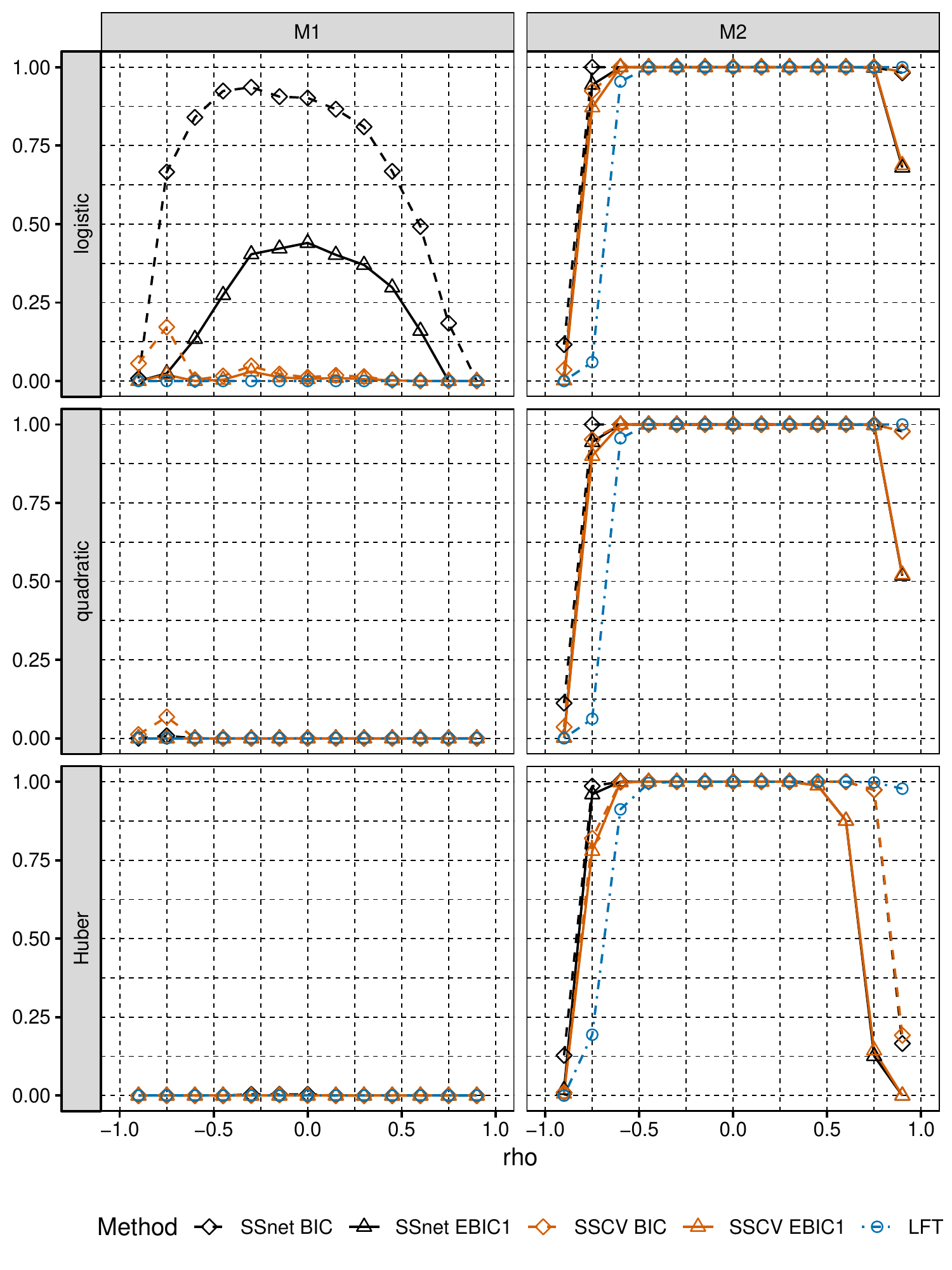}
\caption{$P_{supset}$ for models M1 and M2}
\label{fig_num_exp_M1_M2_P_supset} 
\end{figure}

\begin{figure}[H]
\centering
\includegraphics[width = \textwidth]{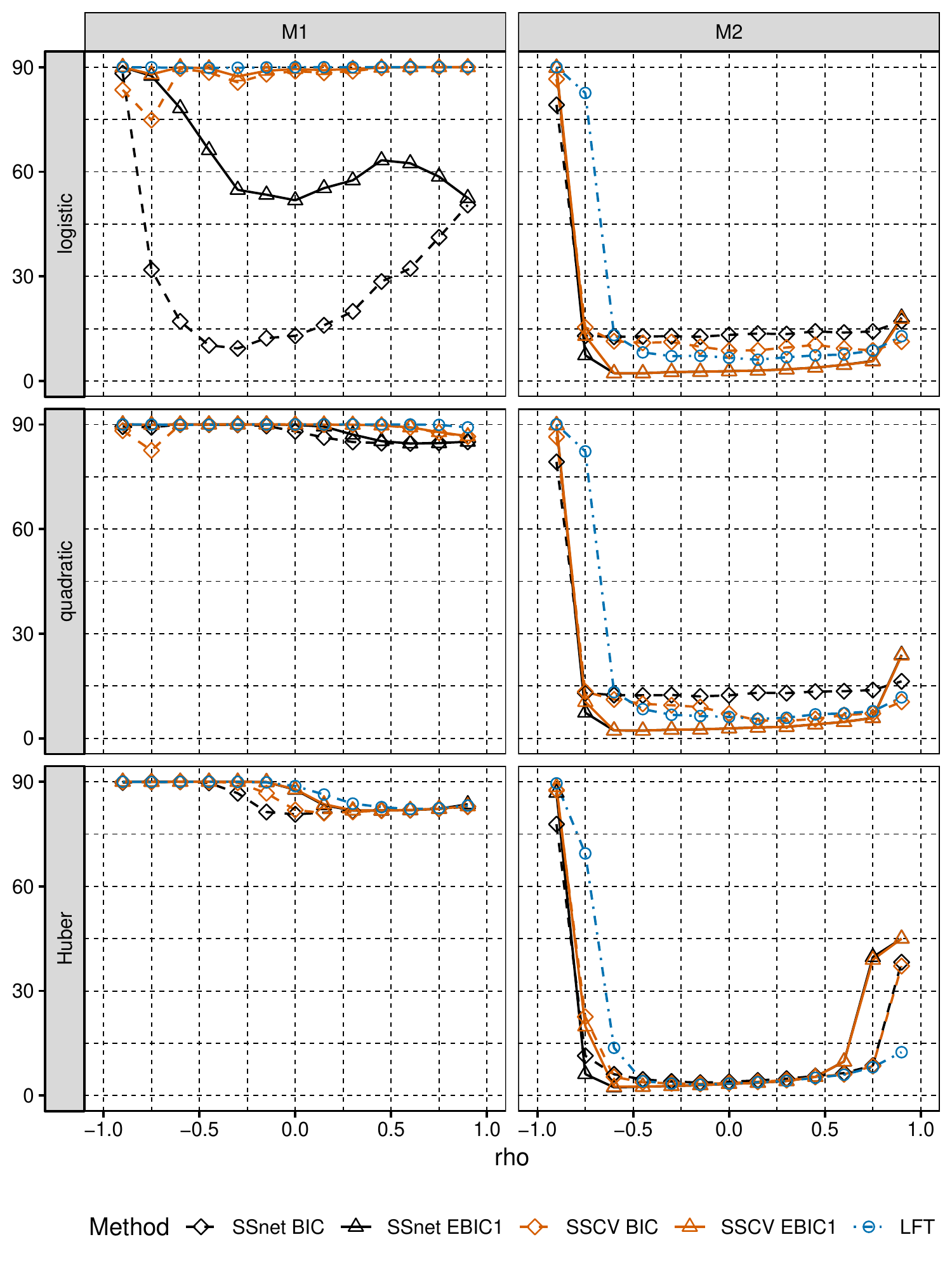}
\caption{$ANGLE$ for models M1 and M2}
\label{fig_num_exp_M1_M2_angle_sparse} 
\end{figure}

\section{Appendix}
Proof of Lemma \ref{lemma_basic_inequality}:
\begin{proof}
Firstly, observe that  function $R_n$ is convex as  $\rho$ is convex. Moreover, from the definition of $\hat{\betab}_L$ we get the inequality:
\begin{equation}\label{lemma_basic_inequality_ineq1}
W_n(\hat{\betab}_L) = R_n(\hat{\betab}_L) - R_n(\betab^*) \le \lambda (||\betab^*||_1 - ||\hat{\betab}_L||_1).
\end{equation}
We note that $\vv - \betab^* \in B_1(r),$ as we have:
\begin{equation}\label{lemma_basic_inequality_intermediate_point_lies_in_ball}
||\vv-\betab^*||_1 = \frac{ ||\hat{\betab}_L - \betab^*||_1}{r + ||\hat{\betab}_L - \betab^*||_1} \cdot r \le r.
\end{equation}
By definition of $W_n,$  convexity of $R_n$, (\ref{lemma_basic_inequality_intermediate_point_lies_in_ball}) and definition of $S$ we have:
\begin{align}
W(\vv) & = W(\vv) -W_n(\vv) + R_n(\vv) - R_n(\betab^*)  \nonumber\\
& \le W(\vv) -W_n(\vv) + u(R_n(\hat{\betab}_L) - R_n(\betab^*)) \le S(r) + u W_n(\hat{\betab}_L). \label{lemma_basic_inequality_ineq2}
\end{align}
From the convexity of $l_1$ norm, (\ref{lemma_basic_inequality_ineq2}), (\ref{lemma_basic_inequality_ineq1}), $||\betab^*||_1 = ||\betab^*_{s^*}||_1$ and triangle inequality it follows that:
\begin{align}
W(\vv) + \lambda ||\vv||_1 & \le W(\vv) + \lambda u ||\hat{\betab}_L||_1 + \lambda(1-u) ||\betab^*||_1 \nonumber\\
& \le S(r) + uW_n(\hat{\betab}_L) + u\lambda (||\hat{\betab}_L||_1 - ||\betab^*||_1) + \lambda||\betab^*||_1 \nonumber\\
& \le S(r) + \lambda ||\betab^*||_1 \le S(r) + \lambda ||\betab^*-\vv_{s^*}||_1 + \lambda ||\vv_{s^{*}}||_1. \label{lemma_basic_inequality_ineq3}
\end{align}
Hence:
\begin{multline*}
W(\vv) + \lambda ||\vv - \betab^*||_1 = (W(\vv) + \lambda ||\vv||_1) + \lambda (||\vv - \betab^*||_1 - ||\vv||_1) \\
\le S(r) + \lambda ||\betab^*-\vv_{s^*}||_1 + \lambda ||\vv_{s^{*}}||_1 + \lambda (||\vv - \betab^*||_1 - ||\vv||_1) = S(r) + 2\lambda ||\betab^*-\vv_{s^*}||_1.
\end{multline*}
\end{proof}

\begin{lemma}\label{lemma_appendix_subgaussian_product_with_bounded_independent}
Assume that $S\sim Subg(\sigma^2)$ and $T$ is random variable such that $|T|\le M,$ where $M$ is some positive constant and $S$ and $T$ are independent. Then $ST \sim Subg(M^2 \sigma^2).$ 
\end{lemma}

\begin{proof}
Observe that:
\begin{equation*}
\EE e^{tST} = \EE ( \EE( e^{tST}|T) ) \le \EE e^{\frac{t^2T^2\sigma^2}{2}} \le e^{\frac{t^2M^2\sigma^2}{2}}.
\end{equation*}
\end{proof}

\bibliography{References}
\bibliographystyle{plainnat}
\end{document}